\documentclass[12pt]{article}

\usepackage{amsmath,amssymb,amsfonts}
\usepackage{amsthm}%
\usepackage{url}%
\usepackage{enumitem}
\usepackage[T1]{fontenc}
\usepackage{mathscinet}

\usepackage{graphicx}
\usepackage{mathrsfs}
\usepackage{enumerate}

\allowdisplaybreaks%

{\theoremstyle{plain}%
 \newtheorem{theorem}{Theorem}
 
 \newtheorem{lemma}{Lemma}%
}
{\theoremstyle{remark}

}
{\theoremstyle{definition}

\newtheorem{example}{Example}
}

\begin{document}

\begin{center}
 {\large On a problem on partitions into distinct primes}

 \ 

{\sc John M. Campbell}

\vspace{0.1in}

{\footnotesize Department of Mathematics and Statistics}

{\footnotesize Dalhousie University}

{\footnotesize Halifax, NS B3H 4R2}

{\footnotesize Canada}

\vspace{0.1in}

{\footnotesize {\tt jh241966@dal.ca}}

\vspace{0.1in}

\end{center}

\begin{abstract}
 For a positive integer $n$, we write $L(n)$ in place of the maximum number of distinct primes whose sum is $n$, assuming there is a 
 way of writing $n$ as a sum of distinct primes, with $L(n) = 0$ otherwise. Again if $n$ admits a partition into distinct primes, then we let 
 $K(n)$ denote the number of parts in the unique partition that has a largest product, with $K(n) = 0$ otherwise. In the latest editions 
 of Guy's classic text on unsolved problems in number theory, the problem as to whether or not $L(n) - K(n)$ is unbounded
 was left open, and it appears that this has remained open. We solve this problem in the affirmative. Our proof relies, to a large extent, on a 
 consequence of Vinogradov's theorem together with a 1948 result due to Sierpi\'nski on  the existence of primes whose gaps to the 
 immediately preceding and following primes both exceed any prescribed positive bound. 
\end{abstract}

\vspace{0.1in}


\vspace{0.1in}


\section{Introduction}\label{secintro}
 The study of partitions of integers into primes is of importance within many areas in additive number theory and in the number-theoretic 
 study of integer partitions. This is evidenced by the roles played by Goldbach-type problems, not to mention research on asymptotic 
 properies concerning the growth of prime partitions. In this paper, we solve, by building on our extensive interactions with GPT-6 
 Astra Pro, a problem that is given in the latest editions of Guy's classic texts on unsolved problems    in number theory \cite[pp.\  
 258--259]{Guy1994} \cite[p.\ 394]{Guy2004}    and that concerns partitions into primes    and that seems to have remained open. 

 Guy \cite[pp.\ 258--259]{Guy1994} \cite[p.\ 394]{Guy2004}    noted that Riddell and Taylor posed the following problem:     Given any  
  partitions of a positive integer $n$ into distinct primes, 
 is it necessarily the case that the unique such partition with a maximal 
 \emph{product} of (distinct) parts is among the partitions of $n$ into (distinct) primes 
 with a maximum \emph{number} of parts? 
 This was answered in the negative, via an explicit counterexample, 
 by Selfridge, and McCranie later determined
 that Selfridge's counterexample is the least counterexample. 
 Guy then asked whether the cardinalities of the two sets suggested in the 
 preceding question can differ by an arbitrarily large amount, and this is clarified below. 
 This question is given as being open in both the latest edition of Guy's text from 2004
 and in the preceding version from 1994. 

 Let $n$ be a positive integer. Suppose that $n$ admits a partition into distinct primes. Among all finite sets of (distinct) primes whose 
 elements sum to $n$, we write $P_n$ to denote the set for which
 the product of the elements is greatest, noting that this is unique by the Fundamental Theorem of Arithmetic. 
 Now, let $S_n$ be any (fixed) finite set that is of maximum cardinality 
 and that consists of (distinct) primes summing to $n$ 
 (noting that the value $|S_n|$ is unique and well defined, 
 regardless of how $S_n$ is chosen according to the given conditions). 
 We then write $K(n) = |P_n|$ and $L(n) = |S_n|$. 
 If $n$ is such that there are no partitions of $n$ into distinct primes, or if $n = 0$, 
 then we let $L(n) = K(n) = 0$. 

 The above problem from Guy's monograph on unsolved problems in number theory
 \cite[\S{F}19]{Guy1994} \cite[\S{F}19]{Guy2004}
 asks (in an equivalent way) whether or not 
 $L(n) - K(n)$ is unbounded. This is equivalent to the proposition 
\begin{equation}\label{quantifyGuy} 
 \forall r \in \mathbb{N} \, \exists n \in \mathbb{N} \ L(n) - K(n) \geq r. 
\end{equation}
 We solve this problem in the affirmative in Section \ref{full} below. 

\subsection{Background}
 The integer sequence $(L(n) : n \in \mathbb{N}_{0})$ defined above agrees with the entry in the On-line Encyclopedia of Integer 
 Sequences\footnote{See \url{https://oeis.org/}.} indexed as {\tt A261153}. It seems that the integer sequence $(K(n) : n \in 
 \mathbb{N}_{0})$, according to any equivalent version of the definition above for this sequence, 
 is not currently indexed in the OEIS. 
 The OEIS entry {\tt A053020}
 is defined, in an equivalent way, 
 as the sequence of integers $n$ 
 for which $K(n)$ and $L(n)$ differ, 
 beginning with $(319, 372, 492, \ldots)$. 
 Based on the extant material in the OEIS entry {\tt A053020}, 
 it seems that Guy's problem, as formulated above, has 
 not been solved previously. 

 A number of past research efforts have concerned 
 topics related to 
 Guy's problem, as formulated above. 
 Notably, Del\'eglise and Nicolas \cite{DelegliseNicolas2013,DelegliseNicolas2015}
 investigated the behavior of 
\begin{equation}\label{DNmax}
 h(n) = \max\left\{ \prod_{p \in A} p \ \big| \ \text{$A$ is a finite set of primes}, \ \sum_{p \in A} p \leq n \right\}. 
\end{equation}
 Observe that: For each positive integer $n$
 admitting a partition into distinct primes, the unique prime set attaining
 $$ \max\left\{ \prod_{p \in A} p \ \big| \ \text{$A$ is a finite set of primes}, \ \sum_{p \in A} p = n \right\}, $$
 as opposed to \eqref{DNmax}, is of cardinality $K(n)$ (according to our definition of $K(n)$ above). 
 Letting $\omega(k)$ denote the number of distinct prime factors of a positive integer $k$, 
 it follows from the work 
 of Del\'eglise and Nicolas \cite[Theorem 6.1 and Corollary 6.1]{DelegliseNicolas2013}
 that $$ \omega(h(n)) = \max_{0 \leq m \leq n} L(m), $$ 
 and Del\'eglise and Nicolas subsequently provided a related 
 identity for $\omega(h(n))$ 
 \cite[Theorem 2]{DelegliseNicolas2015}. Thus the material from Del\'eglise--Nicolas
 is not directly applicable to the Guy problem stated above. 

 Also, Kumar \cite{Kumar2021} 
 studied generating functions and related identities for the norms, 
 or products of parts, of prime partitions in which repeated parts are permitted. 
 The enumeration of partitions into primes has also been studied asymptotically; see, for example, Vaughan’s work \cite{Vaughan2008}.

 For each positive integer \(n\) admitting a partition into 
 distinct primes, the minimum product of such a partition 
 is the term of index \(n\) in the OEIS sequence {\tt A064502}; the entry assigns zero when no such partition exists.

\section{A full solution}\label{full}
 For a positive integer $n$ such that there exists at least one partition of $n$ into distinct primes, 
 the set $P_n$ is one of the prime sets associated with the cardinality maximization 
 used to define $S_n$. Therefore, 
 $K(n) \leq L(n)$, and we see that the same relation holds 
 for the remaining values for nonnegative integers $n$, for which $K(n)$ and $L(n)$ vanish. 

 Let $A$ be a finite set consisting of integers exceeding $1$. We then write $s(A) = \sum_{a \in A} a$ and $v(A) = \prod_{a \in A} a$, with 
 the understanding that $s(\varnothing) = 0$ and that $v(\varnothing) = 1$. 
 An inductive argument can be used to verify that 
\begin{equation}\label{sAleqvA} 
 s(A) \leq v(A). 
\end{equation}
 
 Now, suppose that $A$, as defined above, consists of primes. Then 
 we say that $A$ is   \emph{product-minimal for its sum} or simply \emph{product-minimal} 
 if $v(A) \leq 
 v(B)$ for each finite set $B$ of primes such that $s(B) = s(A)$. 
 Also observe that: For each positive integer $n$
 admitting a partition into distinct primes, 
 a set of primes summing to $n$
 with a minimal product exists and is unique (by the Fundamental Theorem of Arithmetic). 
 
 The following result is given in an equivalent way 
 in a 1948 paper due to Sierpi\'{n}ski  \cite{Sierpinski1948},  
  and this, in turn, was proved using 
 Dirichlet's theorem on primes in arithmetic progressions \cite{Dirichlet1837}. 
 For brevity, we omit repeating Sierpi\'{n}ski's argument. 

\begin{lemma}\label{qHlemma}
 (Sierpi\'nski, 1948) For each positive integer $H$, there are arbitrarily large primes $q > H$ such that $q$
 is the unique prime in 
$ [q-H, q+H]$ \cite{Sierpinski1948}. 
\end{lemma}

\begin{example}
 For $H = 6$, we may verify that the smallest primes $q$ as in 
 Lemma \ref{qHlemma} 
 are as in Table \ref{TableqH}. 
 The integer sequence in the second column in 
 Table \ref{TableqH} agrees with the OEIS sequence {\tt A137870}. 

 \begin{table}[t]
\centering

 \begin{tabular}{ | c | c | c | }
 \hline			
 $H$ & $q$ & $[q-H, q+H]$ \\
 \hline 
 6 & 211 & $[205, 217]$ \\ 
 6 & 293 & $[287, 299]$ \\ 
 6 & 409 & $[403, 415]$ \\ 
 6 & 479 & $[473, 485]$ \\ 
 6 & 631 & $[625, 637]$ \\
 6 & 691 & $[685, 697]$ \\
 6 & 701 & $[695, 707]$ \\
 6 & 709 & $[703, 715]$ \\
 6 & 719 & $[713, 725]$ \\
 6 & 787 & $[781, 793]$ \\
 \hline 
\end{tabular}

 \ 

\caption{An illustration of Lemma \ref{qHlemma}, 
 for $H = 6$ and for the smallest associated values for $q$. }\label{TableqH}
\end{table}

\end{example}

\begin{lemma}\label{lemlargeprimes}
 Let $A$ denote a nonempty and finite set of primes that is product-minimal. 
 Then there are arbitrarily large primes $q > s(A)$
 such that $A \cup \{ q \}$ is product-minimal. 
\end{lemma}

\begin{proof}
 Define $H = 4 v^2(A)$. 
 According to 
 Lemma \ref{qHlemma}, 
 it is possible to select an arbitrarily large prime $q > \max\{ s(A), H \}$
 that is the unique prime in $[q - H, q + H]$. 
 Since each member of $A$ is at most $s(A) < q$, we see that $q \not\in A$. 
 Consequently, the relations
 among 
\begin{equation}\label{sAcupq}
 s(A \cup \{ q \}) = s(A) + q
\end{equation}
 and 
\begin{equation}\label{vAcupq}
 v(A \cup \{ q \}) = q v(A) 
\end{equation}
 hold. 
 So, from the relations in \eqref{sAcupq} and 
 \eqref{vAcupq}, we find that 
 there is at least one finite set $B$ of primes
 such that 
\begin{equation}\label{sBequals} 
 s(B) = s(A) + q
\end{equation}
 and 
\begin{equation}\label{vBleqqM}
 v(B) \leq q \, v(A). 
\end{equation} 
 We then let $B$ be an arbitrary set of this form. 

 We claim that it is necessarily the case that $B = A \cup \{ q \}$. 
 In this direction, we set 
 $R = v(B \setminus \{ \max(B) \})$ and $u = s(B \setminus \{ \max(B) \})$. 
 We thus have that 
\begin{equation}\label{zplusuequals} 
 \max(B) + u = s(A) + q 
\end{equation}
 and 
\begin{equation}\label{zRleqqM}
 \max(B) R \leq q \, v(A). 
\end{equation}
 Of course, each element in $B$ is at most $\max(B)$. 
 This gives us that 
\begin{equation}\label{dplusqleq} 
 s(A) + q \leq |B| \max(B). 
\end{equation}
 A combined application of \eqref{zRleqqM} and 
 \eqref{dplusqleq} gives us that 
\begin{equation}\label{multinequalities}
 R \leq \frac{q \, v(A)}{\max(B)} \leq \frac{|B| q \, v(A)}{s(A) + q} < |B| v(A). 
\end{equation}
 Of course, since $B \setminus \{ \max(B) \}$ has $|B| - 1$ members, 
 each of which exceeds $1$, we find that 
\begin{equation}\label{Rgeq2power}
 R \geq 2^{|B|-1}. 
\end{equation}
 We may verify (using an inductive argument) that 
\begin{equation}\label{squareleqpower}
\ell^2 \leq 2^{\ell+1} 
\end{equation}
 holds for each positive integer $\ell$. 
 A combined application of \eqref{Rgeq2power} and \eqref{squareleqpower} then produces 
\begin{equation}\label{kleq2sqrtR} 
 |B| \leq 2 \sqrt{R}. 
\end{equation}
 From \eqref{multinequalities} and \eqref{kleq2sqrtR} together, we deduce that 
\begin{equation}\label{Rless2Mroot} 
 R < 2 v(A) \sqrt{R}. 
\end{equation}
 Observe that $v(A) > 0$. 
 So, the relation in 
 \eqref{Rless2Mroot} leads us to 
 $ R < 4 v^2(A) = H$. By then exploiting the relation in \eqref{sAleqvA}, we arrive at the relations among
\begin{equation}\label{uleqRlessH} 
 u \leq R < H 
\end{equation}
 and 
\begin{equation}\label{dleqMlessH}
 s(A) \leq v(A) < H. 
\end{equation}
 Using \eqref{zplusuequals}, \eqref{uleqRlessH}, and \eqref{dleqMlessH} together, 
 this gives us that 
\begin{equation}\label{equalsabs} 
 |\max(B) - q| = |s(A) - u| < H. 
\end{equation}
 Now, observe that $\max(B)$ is prime. 
 Also, recall that we chose $q$ as the unique prime 
 in $[q - H, q + H]$, by virtue of Lemma \ref{qHlemma}. 
 So, from \eqref{equalsabs}, 
 we can conclude that the primes $\max(B)$ and $q$ are necessarily equal. 
 So, from \eqref{sBequals} and \eqref{vBleqqM}, respectively, 
 we can deduce that 
 $ s(B \setminus \{ q \}) = s(A) $ 
 and that 
 $ v(B \setminus \{ q \}) \leq v(A)$. 
 The given assumption on $A$ concerning product-minimality allows us 
 to conclude that 
 $ v(B \setminus \{ q \}) = v(A)$, 
 whereas the Fundamental Theorem of Arithmetic 
 gives us that 
 $ B \setminus \{ q \} = A$. 
 Consequently, the equality $B = A \cup \{ q \}$ holds. 
 Since no other (finite) prime set summing to $s(A) + q$ 
 is such that the product of its elements is at most $q \, v(A)$, 
 we obtain the desired result. 
\end{proof}

\begin{example}
 Let $A = \{ 3, 5, 19 \}$. It can be shown that this is product-minimal. 
 Define $d = s(A)$. 
 It can be shown that the smallest prime $q > d$
 such that
 $A \cup \{ q \}$ is product-minimal is 3,323,497. 
 The next smallest such primes are shown in Table \ref{TableAq}. 

\begin{table}[t]
\centering

 \begin{tabular}{ | c | c | }
 \hline			
 $A$ & $q$ \\
 \hline 
 $\{ 3, 5, 19 \}$ & 3,323,497 \\ 
 $\{ 3, 5, 19 \}$ & 6,200,989 \\ 
 $\{ 3, 5, 19 \}$ & 6,579,137 \\ 
 $\{ 3, 5, 19 \}$ & 7,686,541 \\ 
 $\{ 3, 5, 19 \}$ & 8,310,097 \\
 $\{ 3, 5, 19 \}$ & 8,707,537 \\
 $\{ 3, 5, 19 \}$ & 9,675,451 \\
 $\{ 3, 5, 19 \}$ & 10,631,471 \\
 $\{ 3, 5, 19 \}$ & 10,861,867 \\
 $\{ 3, 5, 19 \}$ & 10,894,189 \\
 \hline 
\end{tabular}

 \ 

\caption{An illustration of Lemma \ref{lemlargeprimes}, 
 for $A = \{ 3, 5, 19 \}$ and for the smallest associated values for primes $q$. }\label{TableAq}
\end{table}

\end{example}

\begin{lemma}\label{lemmatd}
 For each integer $t$ exceeding $1$, there exist arbitrarily 
 large odd integers $d$ whose product-minimal prime set is of size $t$. 
\end{lemma}

\begin{proof}
 If $t$ is even, then we begin by setting $A_0 = \{ 2 \}$. 
 If $t$ is odd, then we start with $A_0 = \{ 3 \}$. 
 Observe that both singleton sets are product-minimal for their sums, 
 as each of $2$ and $3$ admit only singleton prime partitions. 

 We then apply Lemma \ref{lemlargeprimes} successively $t-1$ times. 
 With each application of Lemma \ref{lemlargeprimes}, the appended prime is larger than the
 preceding sum, and is thus distinct from all of the previous primes. 
 Also observe that each appended prime is necessarily odd. 

 After completing the above appending procedure, we find that the resulting product-minimal set $A$
 has precisely $t$ elements.
 If $t$ is even, then it consists of $2$ along with $t-1$ odd primes, 
 so the associated sum is odd. 
 If $t$ is odd, it consists of $t$ odd primes, so we again obtain an odd sum. 

 In either case, regardless of the parity of $t$, the final application of 
 Lemma \ref{lemlargeprimes} in the above procedure
 can be taken as being arbitrarily large. So, the resulting odd integer $d = s(A)$ can be
 made to be arbitrarily large. 
\end{proof}

\begin{example}\label{extd}
 For $t = 3$, it may be verified that the smallest odd integer $d$ whose product-minimal prime set
 has exactly $t$ elements is $d = 27$, 
 and the corresponding unique product-minimal prime set is $A = \{ 3, 5, 19 \}$. 
 The next smallest odd integers of the desired form, again for $t = 3$, are shown in 
 Table \ref{Tabletd}. 

\begin{table}[t]
\centering

 \begin{tabular}{ | c | c | c | }
 \hline			
 $t$ & $d$ & $A$ \\
 \hline 
 3 & 27 & $\{ 3, 5, 19 \}$ \\ 
 3 & 35 & $\{ 3, 13, 19 \}$ \\ 
 3 & 51 & $\{ 3, 5, 43 \}$ \\ 
 3 & 57 & $\{ 3, 7, 47 \}$ \\ 
 3 & 65 & $\{ 5, 7, 53 \}$ \\
 3 & 77 & $\{ 3, 7, 67 \}$ \\
 3 & 87 & $\{ 3, 5, 79 \}$ \\
 3 & 93 & $\{ 3, 7, 83 \}$ \\
 3 & 95 & $\{ 5, 7, 83 \}$ \\
 3 & 117 & $\{ 3, 5, 109 \}$ \\
 \hline 
\end{tabular}

 \ 

\caption{An illustration of Lemma \ref{lemmatd}, 
 for $t = 3$ and for the smallest associated values for $d$, 
 with the associated unique product-minimal prime set $A$ being displayed in the final column. }\label{Tabletd}
\end{table}

\end{example}

 \emph{Goldbach's weak conjecture} purports that each odd integer exceeding $5$ can be written as the sum of 
 three not necessarily distinct primes. 
 It is well known that what is referred to as 
 \emph{Vinogradov's theorem} \cite{Vinogradov1937}
 implies that any sufficiently large odd integer can be expressed 
 as a sum of three not necessarily distinct prime numbers. 
 However, we require a refinement of this consequence, 
 as we require the distinctness of the primes in such an expression. 
 
 For a positive integer $N$, define 
\begin{equation}\label{defineR3N}
 R_3(N) = \#\{ (p_1, p_2, p_3) : \text{$p_1, p_2, p_3$ are primes and $p_1 + p_2 + p_3 = N$} \}, 
\end{equation}
 letting it be understood that the tuples involved in 
 \eqref{defineR3N} are ordered and that repeated primes are permitted. 
 Vinogradov proved (cf.\ \cite[Theorem 2]{Vinogradov1937}) that: 
 There are constants $c > 0$ and $N_0 > 0$ such that: For each odd integer $N \geq N_0$, 
 the relation 
\begin{equation}\label{displayVinogradov}
 R_3(N) \geq c \frac{N^2}{(\log N)^3} 
\end{equation}
 holds. This leads us to the following consequence of \eqref{displayVinogradov}. 

\begin{lemma}\label{lemVinogradov}
 Each sufficiently large odd integer is the sum of three pairwise distinct primes (cf.\ \cite{Vinogradov1937}). 
\end{lemma}

\begin{proof}
 For $R_3(N)$ as defined in \eqref{defineR3N}, 
 we find that there exist at most $N$ triples
 of the form indicated in \eqref{defineR3N}
 such that $p_1 = p_2$, since, by fixing a choice for $p_1$, 
 the final entry is determined
 by $p_3 = N - 2 p_1$. 
 We obtain the same bound for triples satisfying
 $p_1 = p_3$ and those satisfying $p_2 = p_3$. 
 So, there are at most $3N$ triples of the form indicated in \eqref{defineR3N}
 containing a repeated prime. 
 Since $\frac{N}{\log^3 N} \to \infty$ as $N \to \infty$, 
 the formulation of Vinogradov's theorem in \eqref{displayVinogradov}
 gives us that 
 $R_3(N) > 3 N$ for all sufficiently large odd $N$. 
 So, at least one triple associated with 
 $R_3(N) $ (again for sufficiently large odd $N$) 
 necessarily has three pairwise distinct entries. 
\end{proof}

 We next require the following result related 
 to the AM-GM inequality. As below, the expression $e$ denotes the familiar constant 
 often referred to as Euler's number. 

\begin{lemma}\label{eleminequality}
 Let $x \geq 2$, and let $G > 0$. 
 Let $U$ and $V$ be finite sets of integers satisfying 
\begin{equation}\label{2lequleqx}
 2 \leq u \leq x
\end{equation}
 for $u \in U$ and 
\begin{equation}\label{conditionG}
 v \geq x + G
\end{equation}
for $v \in V$
 and $V \neq \varnothing$ and 
\begin{equation}\label{sUminussV} 
 s(U) - s(V) = d > 0. 
\end{equation}
 Then the relations in $ \frac{v(U)}{v(V)} \geq \frac{2}{e^3} (d + G) > \frac{d + G}{11} $ hold. 
\end{lemma}

\begin{proof}
 We begin by noting that $|V| \geq 1$, from one of the assumptions given above. 
 Also, we observe that it is necessarily the case that $|U| > |V|$, because, otherwise, 
 we would have that 
 $s(U) \leq |V| x < |V|(x + G) \leq s(V)$, 
 but this contradicts the hypothesis in \eqref{sUminussV}. 
 So, there exists a positive integer $a$ such that 
 the elements in $U$ and $V$ can be written so that 
\begin{equation}\label{orderU} 
 U = \{ u_1 < u_2 < \cdots < u_{a+|V|} \} 
\end{equation}
 and so that 
 $ V = \{ v_1 < v_2 < \cdots < v_{|V|} \}$. By then writing 
\begin{equation}\label{tausumu} 
 \tau = \sum_{i=1}^{a} u_{i}. 
\end{equation}
 The condition in \eqref{sUminussV} then gives us that 
\begin{equation}\label{tauminusd}
 \tau - d = \sum_{j=1}^{|V|} (v_j - u_{a + j}). 
\end{equation}
 The hypothesis associated with \eqref{conditionG} gives us that each 
 difference on the right of \eqref{tauminusd} is at least $G$, 
 and we thus obtain that 
\begin{equation}\label{taugeqdplusbG} 
 \tau \geq d + |V| G \geq d + G. 
\end{equation}
 Observe that $u_{a + j} \geq u_a$ (recalling \eqref{orderU}). 
 Also, from 
 \eqref{orderU} and \eqref{tausumu}, 
 we can deduce that $\tau \leq a u_{a}$. So, this leads us to the relations such that 
\begin{align*}
 \sum_{j=1}^{|V|} \frac{v_{j}}{u_{a+j}} 
 & = |V| + \sum_{j=1}^{|V|} \frac{v_{j} - u_{a + j}}{u_{a+j}} \\ 
 & \leq |V| + \frac{\tau - d}{u_{a}} \\
 & \leq |V| + a.
\end{align*}
 By applying the arithmetic-geometric mean inequality
 to the $|V|$ positive numbers
 of the form $\frac{v_{j}}{u_{a + j}}$ for $j \in \{ 1, 2, \ldots, |V| \}$, 
 and by then applying the elementary inequality 
 $\big( 1 + \frac{a}{|V|} \big)^{|V|} \leq e^a$, we find that 
 $$ \prod_{j=1}^{|V|} \frac{v_{j}}{u_{a+j}} \leq \left( \frac{1}{|V|} \sum_{j=1}^{|V|} \frac{v_{j}}{u_{a + j}} \right)^{|V|}
 \leq \left( 1 + \frac{a}{|V|} \right)^{|V|} \leq e^{a}. $$
 As a consequence, we obtain that 
\begin{equation}\label{prodquot}
 \frac{v(U)}{v(V)} 
 = \frac{ \prod_{i=1}^{a} u_{i} }{ \prod_{j=1}^{|V|} \frac{v_{j}}{u_{a + j}} } 
 \geq e^{-a} \prod_{i=1}^{a} u_{i}.
\end{equation}
 The expressions among $u_1$, $u_2$, $\ldots$, $u_a$
 are pairwise unequal and all exceed $1$ (recalling \eqref{2lequleqx}), 
 and thus $u_{i} \geq i + 1$. Consequently, we have that 
\begin{equation}\label{geqfactorial}
 \prod_{i=1}^{a} u_{i} \geq u_{a} \prod_{i=1}^{a-1} (i+1) 
 = a! u_{a} \geq (a-1)! \tau. 
\end{equation}
 Now, observe that 
 the ratio of consecutive terms 
 in the sequence of expressions of the form 
 $(\alpha - 1)! e^{-\alpha}$ for $\alpha \in \mathbb{N}$ is $\alpha/e$, 
 and thus its minimum occurs at $\alpha = 3$. 
 Consequently, the relation 
\begin{equation}\label{fromratio} 
 \frac{(\alpha-1)!}{e^\alpha} \geq \frac{2}{e^3} 
\end{equation}
 holds for positive integers $\alpha$. 

 Now, a combined application of \eqref{taugeqdplusbG}, 
 \eqref{prodquot}, \eqref{geqfactorial}, and \eqref{fromratio} gives us that 
 $ \frac{v(U)}{v(V)} \geq \frac{(a-1)!}{e^a} \tau \geq 
 \frac{2}{e^3} \tau \geq \frac{2}{e^3} (d + G)$. 
 Since $\frac{2}{e^3} > \frac{1}{11}$, we obtain the desired result. 
\end{proof}

 Now, for a given prime $x$, define
 $ \mathcal{B}_{x} = \{ p : \text{$p$ is prime and $p \leq x$} \}$. 
 Also, we define
 $\sigma_{x} = s(\mathcal{B}_{x})$ and $T_{x} = v(\mathcal{B}_{x})$, 
 and this leads us toward the following Lemma. 

\begin{lemma}\label{lemmaPn}
 We write $A$ in place of the product-minimal prime set of sum $d > 0$, and we then let $M = v(A)$. Now, let $x > d$ be a prime, and 
 let $y$ be the next prime after $x$. 
 Also, let $G = y - x$, 
 and assume that 
\begin{equation}\label{Gstrict}
 G > 11 M.
\end{equation} 
 Also, let $n = \sigma_x -d$. Then $n$ admits a partition into distinct primes, 
 and its unique maximum-product prime set is
 $\mathcal{B}_{x} \setminus A$, which, moreover, 
 equals $P_n$, for $P_n$ as defined in Section \ref{secintro}. 
\end{lemma}

\begin{proof}
 Each element of $A$ is at most $d < x$. 
 We thus have that $A \subseteq \mathcal{B}_{x}$. 
 So, the difference $\mathcal{B}_{x} \setminus A$ consists of distinct primes
 and is such that 
 $s(\mathcal{B}_{x} \setminus A) = \sigma_{x} - d = n$. 
 Also observe that $\mathcal{B}_{x} \setminus A$ is nonempty, since $x \not\in A$. 
 By construction, we also have that 
 $ v(\mathcal{B}_{x} \setminus A) = \frac{T_{x}}{M}$. 
 
 Now, let $E$ denote an arbitrary finite set of primes such that $s(E) = n$. 
 We then write $U = \mathcal{B}_{x} \setminus E$ and $V = E \setminus \mathcal{B}_{x}$. 

 From the above definitions of $U$ and $V$, the set $E$ is obtained from $\mathcal{B}_{x}$
 by removing the elements of $U$ and adjoining the elements of $V$, i.e., so that 
\begin{equation}\label{sEequals}
 s(E) = s(\mathcal{B}_{x}) - s(U) + s(V) = \sigma_{x} - s(U) + s(V). 
\end{equation}
 Also, from the choice of $E$ and the definition of $n$, we   find that   $s(E) = n = \sigma_{x} - d$. This and \eqref{sEequals}    together give  
  us that 
\begin{equation}\label{complementsumidentity}
 s(U) - s(V) = d. 
\end{equation} 
 Our construction also gives that 
\begin{equation}\label{vEandTx}
 v(E) = T_x \frac{v(V)}{v(U)}. 
\end{equation}
 This leads us toward the following case analysis. 

 \ 

\noindent {\bf Case 1:} First suppose that $V = \varnothing$. From \eqref{complementsumidentity}, we obtain that $s(U) = d$. Now, since 
 $A$ is product-minimal  for prime sets summing to $d$, we 
 find that $v(U) \geq M$. So, from \eqref{vEandTx}, we find that 
\begin{equation}\label{vEleq} 
 v(E) \leq \frac{T_{x}}{M}. 
\end{equation}
 Equality in \eqref{vEleq} holds if and only if $v(U) = M = v(A)$. By the Fundamental Theorem
 of Arithmetic, this is equivalent to $U = A$. Since $V = \varnothing$, 
 it follows that equality holds if and only if $E = \mathcal{B}_{x} \setminus A$. 

 \ 

\noindent {\bf Case 2:} Now, suppose that $V$ is nonempty. Observe that each element of $U$ is in $[2, x]$. 
 Each element of $V$ is a prime greater than $x$, and is thus at least
 $y = x + G$. Using Lemma \ref{eleminequality}
 and \eqref{complementsumidentity}, we obtain that 
\begin{equation}\label{justifylast}
 \frac{v(U)}{v(V)} > \frac{d+G}{11} > M, 
\end{equation}
 where the last inequality in \eqref{justifylast} is implied by 
 \eqref{Gstrict}. From \eqref{vEandTx}, we find that 
 $ v(E) < \frac{T_{x}}{M}$. 

 \ 

 So, if $E$ contains a prime greater than $x$, then Case 2 gives $ v(E) < \frac{T_{x}}{M}$, and, if each element of $E$ is at most $x$, 
 then Case 1 gives $ v(E) \leq \frac{T_{x}}{M}$, 
 with equality if and only if $E = \mathcal{B}_{x} \setminus A$. 
 Since $E$ was an arbitrary finite set of primes whose elements sum to $n$ 
 these two cases show that 
 $ v(E) \leq v(\mathcal{B}_{x} \setminus A) = \frac{T_{x}}{M} $ 
 for each such $E$, with equality only when $E = \mathcal{B}_{x} \setminus A$. 
 As above, the elements of $\mathcal{B}_{x} \setminus A$ sum to $n$. So, this set uniquely attains
 the maximum product, and, hence, by the definition of $P_{n}$, 
 we have that $P_n = \mathcal{B}_{x} \setminus A$, as desired. 
\end{proof}

 We are now in a position to prove our main result, as below, 
 and to solve Guy's problem in the affirmative. 

\begin{theorem}
 For each positive integer $r$, there exist infinitely many positive
 integers $n$ admitting partitions into distinct primes for which 
 $L(n) - K(n) \geq r$.
\end{theorem}

\begin{proof}
 Fix a positive integer $r$, and define $t = r + 3$. By Lemma \ref{lemmatd}, we are permitted to choose an odd integer $d$, sufficiently 
 large so that Lemma \ref{lemVinogradov} is applicable, such that the product-minimal prime set $A$ of $d$ has precisely $t$ elements, 
 i.e., so that $s(A) = d$ and $|A| = t$. We then define $M$ so that $M = v(A)$. 
 According to Lemma \ref{lemVinogradov}, there exists a set $C$ consisting
 of three distinct primes satisfying 
 $s(C) = d$ and $|C| = 3$. 
 We henceforth let $d$ and $A$ and $C$ be fixed. 

 Now, we proceed to apply 
 the Sierpi\'nski Lemma, as formulated in Lemma \ref{qHlemma}, letting $H = 11 M$. 
 So, there are infinitely many primes $x$ exceeding $d$ for which $x$ is the unique 
 prime in 
 $ [x - 11M, x + 11M]$. 
 So, for each such $x$, if $y$ is the next prime, then
 $ y - x > 11 M$. 
 We then set $n = \sigma_{x} - d$, with $m = |\mathcal{B}_{x}|$. 
 By Lemma \ref{lemmaPn}, we obtain that 
 $ P_n = \mathcal{B}_{x} \setminus A$. 
 Consequently, the equalities among
\begin{equation}\label{Pnmminust} 
 |P_n| = |\mathcal{B}_{x}| - |A| = m - t 
\end{equation} 
 hold. Each prime in $C$ is at most $d < x$, so that 
 $ C \subseteq \mathcal{B}_{x}$. 
 As a consequence, the prime set
 $ \mathcal{B}_{x} \setminus C $
 is such that 
 $ s(\mathcal{B}_{x} \setminus C) = \sigma_{x} - s(C) = \sigma_{x} - d = n $
 and is such that 
 $ |\mathcal{B}_{x} \setminus C| = m - 3$. 

 Now, recall the definitions of $S_n$ and $P_n$ 
 given in Section \ref{secintro}. By definition, 
 the cardinality of $S_n$ is at least that of every finite set of primes whose elements
 sum to $n$, 
 with 
\begin{equation}\label{cardSngeq} 
 |S_{n}| \geq |\mathcal{B}_{x} \setminus C| = m - 3. 
\end{equation}
 A combined application of \eqref{Pnmminust} and \eqref{cardSngeq}
 then give us that 
\begin{align*}
 L(n) - K(n) & = |S_n| - |P_n| \\ 
 & \geq (m - 3) - (m - t) \\ 
 & = t - 3 \\
 & = r.
\end{align*}

 There are infinitely many eligible primes $x$. 
 Also, recall that 
 $$ \sigma_{x} = s(\mathcal{B}_{x}) = \sum_{\substack{p \leq x \\ \text{$p$ prime}}} p.$$
 For eligible primes $x_1$ and $x_2$ such that $x_1 < x_2$, 
 we have that 
\begin{equation}\label{substackdiffer} 
 \sigma_{x_2} - \sigma_{x_1} = \sum_{\substack{x_1 < p \leq x_2 \\ \text{$p$ prime}}} p. 
\end{equation}
 Since the right-hand side of 
 \eqref{substackdiffer} containsat least the positive term $p = x_2$, since $x_2$ is prime,
 we find that $\sigma_{x_{2}} - \sigma_{x_{1}} \geq x_2 > 0$. 
 So, as $x$ increases through distinct primes, 
 the value $\sigma_x$ strictly increases. 
 So, since $d$ is fixed, the integers of the form $n = \sigma_x - d$ are distinct (as $x$ increases through distinct primes). 
 So, for $r$ as given above, infinitely many positive integers $n$ are such that 
 $L(n) - K(n) \geq r$. Since $r$ was arbitrary, we obtain the desired result. 
\end{proof}

\section{Conclusion}
 We briefly conclude with some further areas of study related to our construction given above. 

 Lemma \ref{qHlemma} and the associated Table \ref{TableqH} raise questions
 in regard to growth properties associated with primes of the form $q$ given in Lemma \ref{qHlemma}, 
 for fixed $H$. 
 In this direction, for a positive integer $H$, 
 and for a positive integer $n$, we let $q_{H}(n)$ denote 
 the $n^{\text{th}}$ smallest prime $q > H$
 such that $q$ is the unique prime in $[q - H, q + H]$. 
 Then, for example, the infinite array 
 $(q_{i}(j))$ begins with 
\begin{equation}\label{infinitearray}
 \left( \begin{matrix}
5 & 7 & 11 & 13 & 17 & 19 & 23 & 29 & 31 & \cdots \\
23 & 37 & 47 & 53 & 67 & 79 & 83 & 89 & 97 & \cdots \\
23 & 37 & 47 & 53 & 67 & 79 & 83 & 89 & 97 & \cdots \\
53 & 89 & 157 & 173 & 211 & 251 & 257 & 263 & 293 & \cdots \\
53 & 89 & 157 & 173 & 211 & 251 & 257 & 263 & 293 & \cdots \\
211 & 293 & 409 & 479 & 631 & 691 & 701 & 709 & 719 & \cdots \\
211 & 293 & 409 & 479 & 631 & 691 & 701 & 709 & 719 & \cdots \\
211 & 293 & 631 & 787 & 797 & 839 & 1249 & 1259 & 1399 & \cdots \\
211 & 293 & 631 & 787 & 797 & 839 & 1249 & 1259 & 1399 & \cdots \\
 \vdots & \vdots & \vdots & \vdots & \vdots & \vdots & \vdots & \vdots & \vdots & \ddots 
 \end{matrix} \right).
 \end{equation}
 Each of the rows of the above array agrees with an extant OEIS sequence (up to offsets). 
 For example, the second row agrees (up to an offset) with the OEIS sequence {\tt A007510}
 consisting of single primes. Similarly, the first column agrees with the OEIS sequence 
 {\tt A102723}, giving the smallest prime $a(n)$ such that $a(n) - x$ and $a(n) + x$, 
 for $x \in \{ 1, 2, \ldots, n \}$, are all composite. 
 Sierpi\'nski \cite{Sierpinski1948} proved, in an equivalent way, 
 that the integer sequence $(q_H(1))_{H \geq 1}$ is indeed infinite (i.e., that $q_H(1)$ does indeed exist for
 each positive integer $H$), 
 but estimating the growth of this sequence is very difficult, since this 
 is closely related to the behavior of prime gaps, i.e., since $q_H(1)$ is the smallest prime whose gaps immediately 
 to the left and right are btoh strictly greater than $H$. 
 The remaining columns of the above array do not appear to currently be in the OEIS, and the main diagonal 
 is not currently in the OEIS, and we encourage the investigation of these integer sequences. 
 
 Lemma \ref{lemlargeprimes} and the associated Table \ref{TableAq} raise questions 
 on the behavior of primes of the form 
 $q$, for a fixed set $A$ of primes that is product-minimal. 
 For fixed $A$, let $q_{A}(n)$ denote the $n^{\text{th}}$
 smallest prime such that $q > s(A)$ 
 and such that $A \cup \{ q \}$
 is product-minimal. 
 What can be said in regard to the growth of 
 $q_{A}(n)$ for arbitrary $n \in \mathbb{N}$ and for fixed $A$? 

 Also, Lemma \ref{lemmatd} and the associated Table \ref{Tabletd} raise questions concerning
 the behavior of the expressions of the forms $t$ and $d$ 
 involved in Lemma \ref{lemmatd}. 
 In this direction, for a fixed integer $t \geq 2$, 
 let $d_t(n)$ denote the $n^{\text{th}}$ smallest odd integer whose product-minimal prime set is of size $t$, 
 and let $A_t(n)$ denote the associated prime set. What can be said, for example, in regard to the behavior of $\max A_t(n)$
 for fixed $t \in \mathbb{N}_{\geq 2}$ and for arbitrary
 $n \in \mathbb{N}$? Or, say, for arbitrary $t \in \mathbb{N}_{\geq 2}$ and for fixed $n \in \mathbb{N}$? 
 What can be said in regard to the behavior of $d_t(n)$ for arbitrary $n \in \mathbb{N}$
 and fixed $t \in \mathbb{N}_{\geq 2}$, or for arbitrary $t \in \mathbb{N}_{\geq 2}$ and fixed $n \in \mathbb{N}$? 

 \subsection*{AI disclosure}
 As indicated above, our above solution to Guy's problem
 is based on our extensive interactions with GPT-6 Astra Pro, and this is described in more 
 detail below. All 
 AI-generated suggestions were substantially revised, corrected, and independently verified by the author, who assumes full responsibility 
 for the mathematical content. The author's use of AI is fully in accordance with editorial standards on the responsible, ethical, and 
 transparent use of AI. 

 On 2026-10-10, the author inputted a prompt into GPT-6 Astra Pro 
 containing an equivalent version of the problem from Guy's text 
 given by determining whether or not \eqref{quantifyGuy} holds
 (with definitions given for $L(n)$ and $K(n)$, and with an explicit attribution to Guy's text on unsolved problems in number theory), 
 concluding with an explicit request that GPT-6 Astra Pro solve the given problem. 
 After {\tt 8m 42s} of thinking, GPT-6 Astra Pro provided a response beginning with 
\begin{verbatim}
I believe the following gives an affirmative proof, though a 
claim settling this problem should be independently checked.
\end{verbatim}
 After extensive clarifications, corrections, revisions, suggestions, verifications, calculations, etc., through extensive 
 discussions between the author and 
 GPT-6 Astra Pro, 
 and after extensive and careful reviews of such discussions on the part of the author, 
 this provided a basis for the above solution, written by the author, 
 to Guy's problem. 

 For example, it was through the author's investigations 
 and computations associated with the infinite array illustrated in \eqref{infinitearray}
 that the author found (through the OEIS entry {\tt A102723}) 
 that the result given above as Lemma \ref{qHlemma} was due (in an equivalent way) to 
 Sierpi\'{n}ski \cite{Sierpinski1948}, but GPT-6 Astra Pro had not previously been aware of this (and also 
 was not able to access to the 
 corresponding journal article). 

\bibliographystyle{plain}
\bibliography{seprefe}

@book {Guy2004,
    AUTHOR = {Guy, Richard K.},
     TITLE = {Unsolved problems in number theory},
    SERIES = {Problem Books in Mathematics},
   EDITION = {Third},
 PUBLISHER = {Springer-Verlag, New York},
      YEAR = {2004},
     PAGES = {xviii+437},
      ISBN = {0-387-20860-7},
   MRCLASS = {11-01 (11-02)},
  MRNUMBER = {2076335},
MRREVIEWER = {Greg\ Martin},
       DOI = {10.1007/978-0-387-26677-0},
       URL = {https://doi.org/10.1007/978-0-387-26677-0},
}

@article{Dirichlet1837,
  author  = {Dirichlet, P. G. L.},
  title   = {{Beweis des Satzes, dass jede unbegrenzte
              arithmetische Progression, deren erstes Glied
              und Differenz ganze Zahlen ohne gemeinschaftlichen
              Factor sind, unendlich viele Primzahlen enth{\"a}lt}},
  journal = {Abhandlungen der K{\"o}niglich Preussischen
             Akademie der Wissenschaften},
  year    = {1837},
  pages   = {45--81}}

@article {Vaughan2008,
    AUTHOR = {Vaughan, R. C.},
     TITLE = {On the number of partitions into primes},
   JOURNAL = {Ramanujan J.},
  FJOURNAL = {Ramanujan Journal. An International Journal Devoted to the
              Areas of Mathematics Influenced by Ramanujan},
    VOLUME = {15},
      YEAR = {2008},
    NUMBER = {1},
     PAGES = {109--121},
      ISSN = {1382-4090,1572-9303},
   MRCLASS = {11P82 (11N37 11P55)},
  MRNUMBER = {2372796},
MRREVIEWER = {Yifan\ Yang},
       DOI = {10.1007/s11139-007-9037-5},
       URL = {https://doi.org/10.1007/s11139-007-9037-5},
}

@book {Guy1994,
    AUTHOR = {Guy, Richard K.},
     TITLE = {Unsolved problems in number theory},
    SERIES = {Problem Books in Mathematics},
   EDITION = {Second},
      NOTE = {Unsolved Problems in Intuitive Mathematics, I},
 PUBLISHER = {Springer-Verlag, New York},
      YEAR = {1994},
     PAGES = {xvi+285},
      ISBN = {0-387-94289-0},
   MRCLASS = {11-01 (00A07 11-02)},
  MRNUMBER = {1299330},
MRREVIEWER = {B.\ M. M. de Weger},
       DOI = {10.1007/978-1-4899-3585-4},
       URL = {https://doi.org/10.1007/978-1-4899-3585-4},
}

@article {Kumar2021,
    AUTHOR = {Kumar, Abhimanyu},
     TITLE = {Analytic study of norms of prime partitions},
   JOURNAL = {Bull. Sci. Math.},
  FJOURNAL = {Bulletin des Sciences Math\'ematiques},
    VOLUME = {171},
      YEAR = {2021},
     PAGES = {Paper No. 103038, 12},
      ISSN = {0007-4497,1952-4773},
   MRCLASS = {05A17 (11A41)},
  MRNUMBER = {4295120},
       DOI = {10.1016/j.bulsci.2021.103038},
       URL = {https://doi.org/10.1016/j.bulsci.2021.103038},
}

@article {Sierpinski1948,
    AUTHOR = {Sierpi\'nski, W.},
     TITLE = {Remarque sur la r\'epartition des nombres premiers},
   JOURNAL = {Colloq. Math.},
  FJOURNAL = {Colloquium Mathematicum},
    VOLUME = {1},
      YEAR = {1948},
     PAGES = {193--194},
      ISSN = {0010-1354,1730-6302},
   MRCLASS = {10.0X},
  MRNUMBER = {28332},
MRREVIEWER = {R.\ D.\ James},
       DOI = {10.4064/cm-1-3-193-194},
       URL = {https://doi.org/10.4064/cm-1-3-193-194},
}

@article{DelegliseNicolas2015,
 author = {Del{\'e}glise, Marc and Nicolas, Jean-Louis},
 title = {On the largest product of primes with bounded sum},
 fjournal = {Journal of Integer Sequences},
 journal = {J. Integer Seq.},
 issn = {1530-7638},
 volume = {18},
 number = {2},
 pages = {article 15.2.8, 35},
 year = {2015},
 language = {English},
 zbMATH = {6427371},
 Zbl = {1378.11013}
}

@article{Vinogradov1937,
 author = {Vinogradov, I. M.},
 title = {Some theorems concerning the theory of primes.},
 fjournal = {Recueil Math{\'e}matique. Nouvelle S{\'e}rie},
 journal = {Rec. Math. Moscou, n. Ser.},
 volume = {2},
 pages = {179--195},
 year = {1937},
 language = {English},
 zbMATH = {2520173},
 JFM = {63.0131.05}
}

@article {DelegliseNicolas2013,
    AUTHOR = {Del\'eglise, M. and Nicolas, J.-L.},
     TITLE = {Maximal product of primes whose sum is bounded},
   JOURNAL = {Proc. Steklov Inst. Math.},
  FJOURNAL = {Proceedings of the Steklov Institute of Mathematics},
    VOLUME = {282},
      YEAR = {2013},
     PAGES = {S73--S102},
      ISSN = {0081-5438,1531-8605},
   MRCLASS = {11N25 (11-04)},
  MRNUMBER = {3457246},
MRREVIEWER = {Harald\ A.\ Helfgott},
       DOI = {10.1134/S0081543813070055},
       URL = {https://doi.org/10.1134/S0081543813070055},
}

\end{document}